\documentclass[11pt,a4paper,reqno]{amsart}
\usepackage{amsmath}
\usepackage{amsfonts}
\usepackage{amssymb}
\usepackage{amscd}
\usepackage{url}
\usepackage{enumerate}
\usepackage[pdftex,bookmarks=true]{hyperref}

\newcommand\Vol{{\operatorname{Vol}}}
\newcommand\rank{{\operatorname{rank}}}

\newcommand\R{{\mathbf{R}}}
\newcommand\C{{\mathbf{C}}}

\renewcommand\P{{\mathbf{P}}}
\newcommand\E{{\mathbf{E}}}

\newcommand\dist{{\operatorname{dist}}}
\newcommand\Z{{\mathbf{Z}}}

\newcommand\ep{\varepsilon}

\newcommand\Ba{{\mathbf a}}

\newcommand\Bc{{\mathbf c}}
\newcommand\Bd{{\mathbf d}}

\newcommand\Bf{{\mathbf f}}

\newcommand\Bp{{\mathbf p}}

\newcommand\Bu{{\mathbf u}}

\newcommand\Bx{{\mathbf x}}
\newcommand\By{{\mathbf y}}
\newcommand\Bz{{\mathbf z}}

\newcommand\BA{{\mathbf A}}

\newcommand\MA{{\mathcal A}}

\newcommand\CM{{\mathcal M}}
\newcommand\CN{{\mathcal N}}

\newcommand\CP{{\mathcal P}}

\newcommand\Ber{\it{Ber}}
\newcommand\Gau{\it {Gau}}

\newcommand\LCD{\mathbf{LCD}}

\newcommand\mD{\mathcal{D}}
\newcommand\mC{\mathcal{C}}
\newcommand\mA{\mathcal{A}}

\theoremstyle{plain}
  \newtheorem{theorem}[subsection]{Theorem}

  \newtheorem{heuristic}[subsection]{Heuristic}
  
  \newtheorem{fact}[subsection]{Fact}
  \newtheorem{lemma}[subsection]{Lemma}
  \newtheorem{corollary}[subsection]{Corollary}
  \newtheorem{question}[subsection]{Question}
  \newtheorem{example}[subsection]{Example}
  
  \newtheorem{remark}[subsection]{Remark}
  
  \newtheorem{claim}[subsection]{Claim}

\theoremstyle{definition}
  \newtheorem{definition}[subsection]{Definition}

\begin{document}
\title{On a condition number of general random polynomial systems}
\author{Hoi H. Nguyen}
\email{nguyen.1261@math.osu.edu}
\address{Department of Mathematics, The Ohio State University, 231 West 18th Avenue, Columbus, OH 43210}

\thanks{The  author is supported by research grant DMS-1358648}

\begin{abstract} Condition numbers of random polynomial systems have been widely studied in the literature under certain coefficient ensembles of invariant type.  In this note we introduce a method that allows us to study these numbers for a broad family of probability distributions. Our work also extends to certain perturbed systems.
\end{abstract}

\maketitle

\section{Introduction}

 \subsection{Condition number of random matrices}\label{subsection:matrix}  Let $\Bf$ be a system of $n$ linear forms  $f_1,\dots, f_{n}$ in $n$ complex variables $\Bx=(x_1,\dots,x_n)\in \C^n$,

$$f_l(\Bx)=a^{(l)}_1x_1+\dots+ a^{(l)}_n x_n, 1\le l\le n.$$

The condition number $\mu(\Bf)$ of $\Bf$ is defined as

$$\mu(\Bf):=\frac{\sigma_1(\Bf)}{\sigma_n(\Bf)},$$

where $\sigma_1(\Bf)$ and $\sigma_n(\Bf)$ are the largest and smallest singular values of $\Bf$. 




An important problem with many practical applications is to bound the condition number of a random matrix. As the largest singular value $\sigma_1$ is well understood, the main problem is to study the lower bound of the least singular value $\sigma_n$.  This problem was first raised by Goldstine and von Neumann \cite{GN}  well back in the 1940s, with connection to their investigation of the complexity of inverting a matrix.

To answer Goldstine and von Neumman's question, Edelman \cite{Edelman}  computed the distribution of the least singular value of  the random matrix $\Bf^{\Gau}$ where $a_{i}^{(l)},1\le i,l\le n$, are iid standard Gaussian. He showed that for all fixed  $\ep >0$

$$ \P( \sigma_n( \Bf^{\Gau})  \leq \ep n^{-1/2}  ) = \int_0^{\ep^2}  \frac{1+\sqrt{x}}{2\sqrt{x}} e^{-(x/2 + \sqrt{x})}\ dx + o(1) =  \ep - \frac{1}{3} \ep^{3 }  +O(\ep^4) +o(1) . $$

Edelman conjectured that this distribution is universal (i.e., it must hold for other distributions of $a_i^{(l)}$, such as Bernoulli.)  Note that the same asymptotic continues to hold for any $ \ep>0$ which can go to $0$ with $n$ (see also \cite{ST,ST1})

\begin{equation}
\P (\sigma_n (\Bf^{\Gau} )  \le \ep n^{-1/2} )  \le \ep.
\end{equation}

Spielman and Teng, in their recent study of smoothed analysis of the simplex method, conjectured that a slightly adjusted bound also holds in the Bernoulli case  \cite{ST}
\begin{equation}
\P (\sigma_n (\Bf^{\Ber} )  \le \ep )  \le \ep n^{1/2} + c^{n},
\end{equation} 

\noindent where $0 < c < 1$ is a constant. The term $c^n$ is needed as $\Bf^{\Ber}$ can be singular with exponentially small probability.

Edelman's conjecture has been proved by Tao and Vu in \cite{TVhard}. This work also confirms Spielman and Teng's conjecture for the case $\ep$ is fairly large ($\ep \ge n^{-\delta}$ for some small constant  $\delta >0$). For $\ep \ge n^{-3/2} $, Rudelson \cite{Rudannal} obtained a strong bound with an extra (multiplicative) constant factor. In a consequent paper, Rudelson and Vershynin \cite{RV} show 

\begin{theorem} \label{theorem:RVsingbound} 
There is a constant $C >0$ and $0 < c <1$ such that for any $\ep>0$, 

$$ \P( \sigma_n (\Bf^{\Ber} ) \le \ep n^{-1/2})  \le C \ep n^{1/2} + c^{n}  . $$ 
\end{theorem} 

This bound is sharp, up to the constant $C$. It also gives a new proof of Kahn-Koml\'os-Szemer\'edi bound \cite{KKSz} on the singularity 
probability of a random Bernoulli matrix.  All these results hold in more general setting, namely that it is enough to assume that the common distribution of the $a_i^{(l)}$ is subgaussian (see \eqref{eqn:subgaussian}) of zero mean and unit variance.

In practice, one often works with random matrices of the type $\Bc+\Bf$ where $\Bc=(c_i^{(l)})$ is deterministic and $\Bf$ has iid entries. For instance, in their works on smoothed analysis, Spielman and Teng used this to model a large data matrix  perturbed by random noise. They proved in \cite{ST} (see also Wschebor \cite{W}) 

\begin{theorem} \label{theorem:STcondition} Let $\Bc=(c_i^{(l)})$ be an arbitrary $n$ by $n$ matrix.
 Then for any $ \ep>0$,
$$\P( \sigma_n (\Bc+\Bf^{\Gau})  \le \ep n^{-1/2}) = O(\ep) . $$
\end{theorem}

One may ask whether there is an analogue of Theorem \ref{theorem:RVsingbound} for this model. The answer is, somewhat surprisingly, negative. However, Tao and Vu managed to prove 

\begin{theorem} \label{theorem:TVsingbound} Assume that $\|\Bc\|_2\le n^\gamma$ for some $\gamma>0$. Then for any $A>0$, there exists $B=B(A,\gamma)$ such that 

$$ \P( \sigma_n (\Bc+ \Bf^{\Ber}) \le n^{-B}) \le n^{-A } . $$ 
\end{theorem}

For  more discussion on this model, we refer to \cite{TVsmooth}. For applications of Theorem \ref{theorem:TVsingbound} in Random Matrix Theory (such as the establishment of the Circular Law) and many related results, we refer to  \cite{NgV} and the references therein.

\subsection{Condition numbers for the study of Newton's method}\label{subsection:Newton} Let $\Bd=(d_1,\dots,d_{n-1})$ be a degree sequence, and $\Bf =\{f_1,\dots,f_{n-1}\}$ be a collection of $n-1$ homogeneous polynomials in $n$ variables of degree $d_1,\dots,d_{n-1}$ respectively, 

$$f_l(x_1,\dots,x_n)=\sum_{\substack{\alpha=(\alpha_1,\dots,\alpha_n)\\ \alpha_1+\dots+\alpha_n = d_l}}  \binom{d_l}{\alpha}^{1/2} a_\alpha^{(l)} \Bx^\alpha,$$

where $\Bx^{\alpha}=x_1^{\alpha_1} \cdots  x_n^{\alpha_n}$.

In their seminal works  \cite{SS1,SS2,SS3,SS4,SS5}, Shub and Smale initiated a systematic study of Newton's method for finding common roots of the $f_i$ over the unit vectors in $\C^n$. 

Define the Weyl-norm of the system $\Bf$ by $\|\Bf\|_W:=\sqrt{\|f_1\|_W^2 +\dots +\|f_{n-1}\|_W^2}$, where $\|f_l\|_W^2:= \sum_\alpha |a_\alpha^{(l)}|^2$. 
For each complex unit vector $\Bx=(x_1,\dots,x_n)$ in $S^{n-1}$, we measure the singularity of the system at $\Bx$ by

$$\mu_{complex}^{(1)}(\Bf,\Bx) = \|\Bf\|_W \times \|(D_\Bx|_{T_\Bx})^{-1} \Delta \|_2,$$

where $D_{\Bx}|_{T_\Bx}$ is the Jacobian of the system $\Bf$ restricted to the tangent space at $\Bx$, and $\Delta$ is the diagonal matrix of entries $(\sqrt{d_l},1\le l\le n-1)$.

We denote the condition number of the system by

$$\mu_{complex}^{(1)}(\Bf)= \sup_{\Bx\in S^{n-1}, f_1(\Bx)=\dots =f_{n-1}(\Bx)=0}\mu_{complex}^{(1)}(\Bf,\Bx).$$

To analyze the effectiveness of Newton's method for finding commons roots of the $f_i$, Shub and Smale show that, under an invariant probability measure, the condition number of $\Bf$ is small with high probability.

\begin{theorem}\cite{SS2,Kostlan}\label{theorem:SS2} 
Assume that the coefficients $a_\alpha^{(l)}$ are iid standard complex-Gaussian random variables, then

$$\P(\mu_{complex}^{(1)}(\Bf^{\Gau}) >1/\ep)  =O( n^4N^2\mathcal{D} \ep^4).$$
\end{theorem}

Here $\mD:=\prod d_i$ is the Bezout number and $N:=\sum_{i=1}^{n-1} \binom{n-1+d_i}{d_i}$.



Beside finding common complex roots, another important problem is to find common real roots. In a recent series \cite{CKMW1,CKMW2,CKMW3}, Cucker, Krick, Malajovich and Wschebor have studied this problem in detail.  For convenience, Cucker et. al. introduced the following condition number. 

For any $\Bx\in \R^n$, we measure the singularity of the system at $\Bx$ by
$$\mu_{real}^{(2)}(\Bf,\Bx) = \min \left\{ \sqrt{n} \max_{i} \|f_i\|_W \times \|(D_\Bx|_{T_\Bx})^{-1} \Delta \|_2, \frac{\max_{i} \|f_i\|_W}{ \max_{i} |f_i(\Bx)|} \right \}.$$

The condition number of the system is then defined as 

$$\mu_{real}^{(2)}(\Bf): =\sup_{\Bx\in \R^n, \|\Bx\|_2=1} \mu_{real}^{(2)} (\Bf,\Bx).$$

Notice that the definition of $\mu^{(2)}$ is taken over all $\|\Bx\|_2=1$, and thus (with restricted to $\R^n$) is more general than $\mu^{(1)}$. We recite here a key estimate by Cucker, Krick, Malajovich and Wschebor with respect to $\mu^{(2)}$.

\begin{theorem}\cite{CKMW3}\label{theorem:CKMW3} Assume that $a_\alpha^{(i)}$ are iid standard real Gaussian random variables, then
$$\P(\mu_{real}^{(2)}(\Bf^{\Gau}) >1/\ep) =O\Big( \max_i d^2_i \sqrt{\mD} \sqrt{N}n^{5/2} \sqrt{n} \ep \sqrt{\log \frac{1}{\ep \sqrt{n}}}\Big),$$

provided that $\ep^{-1} =\Omega(\max_i d^2_i n^{7/2} N^{1/2})$.
\end{theorem}

Roughly speaking (see for instance \cite{CKMW1} or \cite[Section 19]{BC}), Cucker, Krick, Malajovich and Wschebor showed that there exists an iterative algorithm that returns the number of real zeros of $\Bf$ and their approximations and performs $O(\log (nD \mu_{real}^{(2)} (\Bf)))$ iterations with a total cost of 

$$O\left(\large [C(n+1) D^2 (\mu_{real}^{(2)})^2\large ]^{2(n+1)} N \log (n D\mu_{real}^{(2)}(\Bf)\right ).$$ 

Henceforth, the probabilistic analysis of $\mu_{real}^{(2)}$, Theorem \ref{theorem:CKMW3}, plays a key role in their study. 

The proofs of Theorem \ref{theorem:SS2} and Theorem \ref{theorem:CKMW3}, on the other hand, heavily rely on the invariance property of (real and complex) Gaussian distributions, and are extremely involved. 


Motivated by the results discussed in Subsection \ref{subsection:matrix}, it is natural and important to study the condition numbers $\mu_1$ and $\mu_2$ for polynomial systems under more general distributions such as Bernoulli. This problem is also closely related to a question raised by P.~Burgisser and F.~Cucker in \cite[Problem 7]{BC}. 

Roughly speaking, there are two main technical obstacles of our task: first is the absence of invariance property of distributions and second is the lacking of linear algebra tools (compared to the condition number problem of matrices discussed in Subsection \ref{subsection:matrix}). As a result, to our best knowledge, even the following  simple and natural question is not even known.

\begin{question}\label{question:simple}
Assume that $a_\alpha^{(l)}$ are iid Bernoulli random variables (taking value $\pm 1$ with probability 1/2). Is it true that with probability tending to 1 (as $n \rightarrow \infty$), there does not exist non-zero vector $\Bx\in \R^n$ (or $\Bx\in \C^n$) with $\Bf(\Bx)=0$ and $\rank(D_\Bx|_{T_\Bx})<n-1$?
\end{question}


\subsection{Our result} To simplify our work, we will be focusing only on the Kostlan-Shub-Smale model where $n$ is sufficiently large and  $d_i=d\ge 2$ for all $i$. (Note that the case $d_i=1$ corresponds to rectangular matrices, the reader is invited to consult for instance \cite{RV-rec} for related results.) For this uniform system, Theorem \ref{theorem:SS2} and Theorem \ref{theorem:CKMW3} read as follows.

\begin{theorem}[Non-degeneration of uniform homogenous polynomial systems]\label{cor:multi} Assume that $c_\alpha$ are iid standard complex Gaussian, then 

$$\P(\mu^{(1)}(\Bf^{Gau}) >1/\ep) =O \left((n+d)^{O(d)}(d^{n/4}\ep)^4 \right).$$

Moereover, if $c_\alpha$ are iid standard real Gaussian random variables, then
 
$$\P(\mu_{real}^{(2)} (\Bf^{\Gau})>1/\ep) =O\left(n^{O(d)} d^{n/2}  \ep \sqrt{\log \frac{1}{\ep}}\right).$$ 

\end{theorem}

Notice that these bounds are effective only when $\ep$ is exponentially small, namely $\ep \ll d^{-n/4}$ in the complex case and $\ep \ll d^{-n/2}$ in the real case (these are the right scaling as the variance of a typical coefficient is $d$). A closer look at Theorem \ref{cor:multi} reveals the following.


\begin{heuristic}\label{heuristic:1}
With high probability, for any $\Bx\in S^{n-1}$, $(\|(D_\Bx|_{T_\Bx})^{-1}\|_2)^{-1}$ and $\|\Bf^{Gau}(\Bx)\|_2$ cannot be too small at the same time.  In other words, such a random system is not ``close'' to having ``double roots" with high probability.
\end{heuristic}


Although our method can be extended to the complex case, we will be mainly focusing on the real roots to simplify the presentation.  Furthermore, as $\mu^{(2)}$ is more general than $\mu^{(1)}$, we will be limited ourself to a quantity similar to  $\mu^{(2)}$ only. 

Let $d\ge 2$ be an integer. Let $\mC=\{c^{(l)}_{i_1\dots i_d},0\le i_1,\dots, i_d \le n, 1\le l \le n-1\}$ be a deterministic system. We consider a random array $\mA=\{a^{(l)}_{i_1\dots i_d},0\le i_1,\dots,i_d \le n, 1\le l \le n-1\}$, where $a^{(l)}_{i_1\dots i_d}$ are iid copies of real random variable $\xi$ with mean zero, variance one, and there exists $T_0>0$ such that

\begin{equation}\label{eqn:subgaussian}
\forall t>0 \quad \P(|\xi|\ge t)=O(\exp(-t^2/T_0)).
\end{equation}

Such subgaussian distributions clearly cover Gausssian and Bernoulli random variables as special cases. 

For $\Bx=(x_1,\dots,x_n)\in S^{n-1}$ of $\R^n$, we consider a system $\Bf=(f_1,\dots,f_{n-1})$ of $n-1$ $d$-linear forms 

\begin{align*}
f_l(\Bx)&:= \sum_{1\le i_1, \dots, i_d \le n} c^{(l)}_{i_1\dots i_d}  x_{i_1}\dots x_{i_d}  + \sum_{1\le i_1, \dots, i_d \le n}  a^{(l)}_{i_1\dots i_d}  x_{i_1}\dots x_{i_d}\\
&:=f_{l, det}(\Bx) + f_{l,rand}(\Bx).
\end{align*}

In particular, if $\xi$ is the standard Gaussian and the deterministic system vanishes, then for any ordered $d$-tuples $\alpha=\{i_1\le \dots \le i_d\}$, the coefficient of $\Bx_\alpha=x_{i_1}\dots x_{i_d}$ is a sum of $\binom{d}{\alpha}$ iid copies of $\xi$, which in turn can be written as $\sqrt{\binom{d}{\alpha}} \xi_{\alpha}$ with a standard Gaussian variable $\xi_\alpha$. This is exactly the model considered by Cucker et. al. as above. Recall that for $\Bx\in \R^n$, the Jacobian matrix $D_\Bx$ of $\Bf$ at $\Bx$ is given by

$$D_\Bx=\left(\frac{\partial f_l(\Bx)}{\partial x_j}\right)_{1\le l\le n-1,1\le j\le n}.$$

For $1\le l\le n-1$, the gradient of $f_{l}$ at $\Bx$ is 

$$D_{l,\Bx}^{(1)}= \left(\frac{\partial f_{l}}{\partial x_1},\dots, \frac{\partial f_{l}}{\partial x_n}\right)$$ 

while  the Hessian is

$$D_{l,\Bx}^{(2)} = \left(\frac{\partial^2 f_{l}}{\partial x_i \partial x_j}\right)_{1\le i,j\le n}.$$  

In general for $0\le k\le d$, $D_{l,\Bx}^{(k)}$, the $k$-th order derivative, is the $k$-multilinear form

 $$D_{l,\Bx}^{(k)} = \left(\frac{\partial^k f_{l}}{\partial x_{i_1} \dots  \partial x_{i_k}}\right)_{1\le i_1,\dots,i_k \le n}.$$  

Define similarly  $D_{l,\Bx, det}^{(k)},D_{l,\Bx,rand}^{(k)}$ for the deterministic and random systems respectively. 

To control the smallness of  $(\|(D_\Bx|_{T_\Bx})^{-1}\|_2)^{-1}$ and $\|\Bf(\Bx)\|_2$ simultaneously, motivated by  \cite[p.220]{CKMW3}, we introduce a function $L(\Bx,\By)$ for $\Bx\perp \By$ as follows

$$L(\Bx,\By)=  \sqrt{\frac{\|\Bf(\Bx)\|_2}{(d^{9/2}n)^{1/2}}   +  \frac{\|D_{\Bx}(\By)\|_2^2}{d^{9/2}n}}.$$
 
Let $L$ be the minimum value that $L(\Bx,\By)$ can take,

$$L:= \min_{\Bx,\By \in S^{n-1}, \Bx\perp \By} L(\Bx,\By).$$

Our first main goal is to show that $L$ cannot be too small with high probability.

\begin{theorem}[Main theorem, homogeneous system] \label{theorem:main'}Assume that all the coefficients $a_{i_1\dots i_d}^{(l)}$ are iid copies of a random variable $\xi$ satisfying \eqref{eqn:subgaussian}. Then there exist positive constants $K_0\ge 1$ and $c_0$ depending only on $\xi$  with $0<c_0<1$ such that 
$$\P(L \le \ep) \le K_0^n d^{9n/4} \ep + c_0^n$$
for all $\ep>0$ and all $2\le d\le n^{\ep_0}$, with $\ep_0$ a sufficiently small constant again depending only on $\xi$.
\end{theorem}

We remark that the ``error term" $c_0^n$ in Theorem  \ref{theorem:main'} is not avoidable in general. 

\begin{example}\label{example:1} With $d=2$ and $\P(\xi =\pm 1)=1/2$, it is easy to check that $\P(\Bf(\Bx_0=0 \wedge D_{\Bx_0}|T_{\Bx_0} \mbox{ is 
singular} )) = \Omega((3/8)^{-2n})$, where $\Bx_0=(1,1,0,\dots,0)$.  
\end{example}

As a consequence of Theorem \ref{theorem:main'}, one confirms Question \ref{question:simple} and Heuristic \ref{heuristic:1} for a wide range of coefficient distributions.

\begin{corollary}\label{cor:1} With the same assumption as in Theorem \ref{theorem:main}, we have
\begin{itemize}
\item \textup{(}Non-existence of  ``double roots" for random discrete systems\textup{)} 
\begin{equation}\label{eqn:1}
\P\left(\exists \Bx , \By \in S^{n-1}, \Bx \perp \By \wedge  f(\Bx)=0 \wedge D_\Bx(\By)=0 \right)\le c_0^n,
\end{equation}
\vskip .1in
\item \textup{(}Regularity at roots and non-vanishing at critical points \textup{)} 
\begin{align}\label{eqn:2}
\max \Big \{ & \P \left(\exists \Bx, \By \in S^{n-1}, \Bx \perp \By, f(\Bx)=0 \wedge \|D_\Bx(\By)\|_2  \le  d^{9/4} \sqrt{n} \ep \right), \nonumber \\
& \P\left(\exists \Bx , \By \in S^{n-1}, \Bx \perp \By, D_\Bx(\By)=0 \wedge \|\Bf(\Bx)\|_2 \le d^{9/8}n^{1/4}\ep^2 \right) \Big \} \nonumber \\   
& \le  K_0^n d^{9n/4}\ep  + c_0^n,
\end{align}
\vskip .1in
\item \textup{(}Simultaneous vanishing\textup{)}  
\begin{align}\label{eqn:3}
&\P\left(\exists \Bx , \By \in S^{n-1}, \Bx \perp \By, \|\Bf(\Bx)\|_2 \le (d^{9/2}n)^{1/4} \ep  \wedge \|\Bf(\By)\|_2 \le (d^{9/2}n)^{1/4} \ep \right)\nonumber \\
&\le K_0^n d^{9n/4}\ep^{1/2}  + c_0^n,
\end{align}
\end{itemize}
\end{corollary}

where  in the last estimate  we replaced $\ep^2$ by $\ep$ (together with some very generous estimates on $\|D_\Bx(\By)\|_2$).

As noted by Example \ref{example:1}, \eqref{eqn:1} is optimal (with respect to exponential decay). Moreover, the RHS of \eqref{eqn:3} is comparable to the result of Cucker et. al. from Theorem \ref{cor:multi} in the regime that $d$ is sufficiently large and $d\le n^{\ep_0}$. Our proof shows that the error term $c_0^n$ from Theorem \ref{theorem:main} is felt at "sparse" vectors (such as $\Bx_0$ from Example \ref{example:1}). 

More importantly, our method extends to perturbed systems under appropriate assumptions upon  the deterministic system $\mC$.

\begin{definition} We say that the deterministic system $\mC$ is $\gamma$-controlled if
\begin{align}\label{eqn:deterministic}
\max\Big(&\sup_{\Bx \in S^{n-1}}\|\Bf_{det}(\Bx)\|_2^2,\sup_{\Bx,\By_1 \in S^{n-1}}\|D_{\Bx,det}^{(1)}(\By_1)\|_2^2,\dots,\nonumber \\
&\sup_{\Bx, \By_1,\dots,\By_d\in S^{n-1}} \|D_{\Bx,det}^{(d)}(\By_1,\dots,\By_d)\|_2^2 \Big ) \le n^\gamma,
\end{align}
\end{definition}

where $\|D_{\Bx,det}^{(1)}(\By_1)\|_2 = \sqrt{\sum_{1\le l\le n-1} (D_{l,\Bx,det}^{(1)}\By_1^T)^2}$ and so on.


\begin{theorem}[Main theorem, perturbed systems] \label{theorem:main}Assume that $\mC$ is a deterministic system satisfying \eqref{eqn:deterministic} with $\gamma \le 19/18$ and that all the coefficients $a_{i_1\dots i_d}^{(l)}$ are iid copies of a random variable $\xi$ satisfying \eqref{eqn:subgaussian}. Then there exist positive constants $K_0$ and $c_0$ depending only on $\xi$ and $\gamma$ with $0<c_0<1$ such that 
$$\P(L \le \ep) \le K_0^n (d^{9/4} + n^{\gamma/2-1/2})^{n}\ep + c_0^n$$
for all $2\le d\le n^{\ep_0}$ with $\ep_0$ a sufficiently small absolute constant depending on $\xi$ and $\gamma$.
\end{theorem}

We have not tried to optimize the constant $19/18$ on $\gamma$, but our method does not seem to extend to the whole $\gamma=O(1)$ regime. On the other hand, the result remains valid if we assume $d$ sufficiently large depending on $\gamma$ (see Remark \ref{remark:improvement}). 

We believe that our result will be useful for the study of universality problems for roots and critical points of general random polynomial systems. The reader is invited to consult for instance \cite[Lemma 6]{NgOV} for a recent application of this type for univariate random polynomials.




The rest of the note is organized as follows. The main ideas to prove Theorem \ref{theorem:main} is introduced in Section \ref{section:main:ideas}. Sections \ref{section:incompressible},  \ref{section:operatornorm} and \ref{section:compressible} will be devoted to prove the main ingredients subsequently.

\section{Proof of Theorem \ref{theorem:main}: the ideas}\label{section:main:ideas}


Our treatment will be for general $\gamma=O(1)$. The upper bound of $\gamma$ will be required at the end of Section \ref{section:incompressible}. As there is nothing to prove if $\ep>d^{-9n/4}$, we will assume  $\ep\le d^{-9n/4}$. We will verify Theorem \ref{theorem:main} for 
\begin{equation}\label{eqn:assumption:ep}
n^{-(\gamma/2 -17/36)n}\le \ep \le d^{-9n/4}.
\end{equation}

The result for $\ep\le n^{-(\gamma/2-17/36) n}$ easily follows as $K_0^n(d^{9/4}+n^{\gamma/2-1/2})^n n^{-(\gamma/2-17/36) n}=o(c_0^n)$, provided that $\ep_0$ is sufficiently small and $n$ is sufficiently large.

\subsection{Growth of function}
First of all, we will invoke the following bound.  

\begin{theorem}\label{theorem:operatornorm}
Assume that $\xi$ has zero mean, unit variance, and satisfies \eqref{eqn:subgaussian}. Then there exists an absolute positive constant $C_0$ independent of $d$ such that the following holds with probability at least $1-\exp(-dn/2)$

\begin{align}\label{eqn:operatornorm1}
\max\Big(\sup_{\Bx \in S^{n-1}}\|\Bf_{rand}(\Bx)\|_2^2,\sup_{\Bx,\By_1 \in S^{n-1}}\|D_{\Bx,rand}^{(1)}(\By_1)\|_2^2, & \sup_{\Bx,\By_1,\By_2 \in S^{n-1}}\|D_{\Bx,rand}^{(2)}(\By_1,\By_2)\|_2^2 \Big)\le C_0 d^{9/2} n,
\end{align}

and 

\begin{align}\label{eqn:operatornorm2}
\max\Big(\sup_{\Bx,\By_1,\By_2,\By_3 \in S^{n-1}}\|D_{\Bx,rand}^{(3)}(\By_1,\By_2,\By_3)\|_2^2 , \dots , &\sup_{\Bx, \By_1,\dots,\By_d\in S^{n-1}} \|D_{\Bx,rand}^{(d)}(\By_1,\dots,\By_d)\|_2^2 \Big ) \le n^{\omega(d)}.
\end{align}

\end{theorem}







Notice that \eqref{eqn:operatornorm2} is rather straightforward because $d\le n^{\ep_0}$, and without affecting the probability much, one can assume that all of the coefficients $a_{i_1\dots i_d}^{(l)}$ are bounded by $n^{O(1)}$. The proof of the less trivial estimate, \eqref{eqn:operatornorm1}, will be presented in Section \ref{section:operatornorm}. 

Together with condition  \eqref{eqn:deterministic} of $c_{i_1\dots i_d}$ and by the triangle inequality, we obtain a similar bound for the perturbed system $\Bf=\Bf_{det} +\Bf_{rand}$.

\begin{theorem}\label{theorem:together} With probability at least $1-\exp(-dn/2)$, the following holds

\begin{align}
\max\Big(\sup_{\Bx \in S^{n-1}}\|\Bf(\Bx)\|_2^2,\sup_{\Bx,\By_1 \in S^{n-1}}\|D_{\Bx}^{(1)}(\By_1)\|_2^2, & \sup_{\Bx,\By_1,\By_2 \in S^{n-1}}\|D_{\Bx}^{(2)}(\By_1,\By_2)\|_2^2 \Big) \le C_0 d^{9/2} n + n^\gamma 
\end{align}

and 

\begin{align}\label{eqn:operatornorm2:together}
\max\Big(\sup_{\Bx,\By_1,\By_2,\By_3 \in S^{n-1}}\|D_{\Bx}^{(3)}(\By_1,\By_2,\By_3)\|_2^2 , \dots , &\sup_{\Bx, \By_1,\dots,\By_d\in S^{n-1}} \|D_{\Bx}^{(d)}(\By_1,\dots,\By_d)\|_2^2 \le n^{\omega(d)}.
\end{align}

\end{theorem}

Next, we translate the assumption of $L \le \ep$ into slow growth of $\Bf$.

\begin{claim}[Growth of function] \label{claim:growth} With probability at least $1- \exp(-dn)$, the following holds. Assume that $L(\Bx,\By)\le\ep$ for some $\Bx,\By \in S^{n-1}$ with $\Bx \perp \By$, then for any $t\in \R$ with $|t|\le 1$\ and any $\Bz\in \R^n$ with $\|\Bz\|_2\le 1$,
\begin{equation}\label{eqn:growth}
\|\Bf(\Bx+\ep t\By + \ep^2 \Bz)\|_2 \le C_0' (d^{9/4}+n^{\gamma/2-1/2}) \sqrt{n}\ep^2,
\end{equation}
where $C_0'$ is an absolute constant.
\end{claim}

\begin{proof}(of Claim \ref{claim:growth}) We condition on the events considered in Theorem \ref{theorem:operatornorm} and Theorem \ref{theorem:together}. First of all, for each $1\le l\le n-1$, by Taylor expansion

$$f_l(\Bx +\ep t \By + \ep^2 \Bz)= f_l(\Bx) + D_{l,\Bx}^{(1)}(\ep t \By +\ep^2 \Bz)^T +  \frac{1}{2} (\ep t \By + \ep^2 \Bz) D_{l,\Bx}^{(2)}(\ep t \By + \ep^2 \Bz)^T+ o(\ep^2),$$

where we used \eqref{eqn:operatornorm2:together} for the remainder, noting that $\ep \le d^{-9n/4}$.

By the triangle inequality, 

\begin{align*}
|f_l(\Bx +\ep t \By + \ep^2 \Bz)| &\le |f_l(\Bx)|+ \ep |D_{l,\Bx}^{(1)}\By^T| +   \ep^2 |D_{l,\Bx}^{(1)}\Bz^T| + \frac{1}{2}\ep^2 |(t \By +\ep \Bz) D_{l,\Bx}^{(2)}(t \By +\ep \Bz)^T| + o(\ep^2)| \\
&\le  |f_l(\Bx)|+ \ep |D_{l,\Bx}^{(1)}\By^T| +   \ep^2 |D_{l,\Bx}^{(1)}\Bz^T| + \ep^2 (1+\ep^2) |\Bu D_{l,\Bx}^{(2)}\Bu^T|  + o(\ep^2),
\end{align*}

where $\Bu:=(t \By +\ep \Bz)/\sqrt{2(t^2+\ep^2)}$ (and hence $\|\Bu\|_2\le 1$). 

By Theorem \ref{theorem:together}, $\sum_l |D_{l,\Bx}^{(1)}\Bz^T|^2$ and  $\sum _l |\Bu D_{l,\Bx}^{(2)}\Bu^T|^2$ are smaller than $2(C_0d^{9/2} n+n^\gamma)$. As such, by Cauchy-Schwarz inequality

\begin{align*}
\sum_l f_l^2(\Bx +\ep t \By + \ep^2 \Bz) & \le 4\sum_l  f_l^2(\Bx) + 4 \ep^2  \sum_l (D_{l,\Bx}^{(1)}\By^T)^2  \\ 
&+  4\ep^4  \sum_l (D_{l,\Bx}^{(1)}\Bz^T)^2  + 4\ep^4 \sum _l (\Bu D_{l,\Bx}^{(2)}\Bu^T)^2  + o(\ep^4),\\
&\le 4 d^{9/2}n \ep^4  + 4d^{9/2}n\ep^4 +8(C_0d^{9/2}n+n^\gamma) \ep^4 + 8(C_0d^{9/2}n+n^\gamma) \ep^4 +o(\ep^4),
\end{align*}

where we used the assumption that  

$$\sum_l  |f_l(\Bx)|^2 =\|\Bf(\Bx)\|_2^2 \le  d^{9/2} n L^4\le d^{9/2} n\ep^4$$ 

and 

$$\sum_l |D_{l,\Bx}^{(1)}\By^T |^2 = \|D_{\Bx}(\By)\|_2^2 \le d^{9/2} n L^2 \le d^{9/2}n\ep^2.$$

Thus $$\|f(\Bx+\ep t \By +\ep^2 \Bz)\|_2 \le C_0'(d^{9/4}+n^{\gamma/2-1/2})\sqrt{n}\ep^2. $$
\end{proof}

Notice that as $\langle \Bx,\By \rangle =0$, the distance from $\Bx+\ep t \By + \ep^2 \Bz$ to $S^{n-1}$ is at most  $2\ep^2$, and so

$$\Bx+\ep \By + \ep^2 \Bz \in  S_{\ep^2}:=S^{n-1}+B(0,2\ep^2).$$ 

With this notation, because the set $\{\Bx+\ep t \By +\ep^2 \Bz, \|\Bz\|_2 \le 1, |t|\le 1\}$ has volume at least $\frac{\pi^{n/2}}{\Gamma(n/2+1)} \ep^{2(n-1)+1}$, \eqref{eqn:growth} implies that there exists $A\subset S_{\ep^2}$  with volume at least  $\frac{\pi^{n/2}}{\Gamma(n/2+1)} \ep^{2(n-1)+1}$ such that $\|f(\Ba)\|_2\le  C_0'(d^{9/4}+n^{\gamma/2-1/2})\sqrt{n}\ep^2 $ for all $\Ba\in A$. Thus, in order to prove Theorem \ref{theorem:main} it suffices to show the following.

\begin{theorem}\label{theorem:measure} There exist $K_0,c_0$ such that the following holds 

\begin{align*}
\P\Big(\exists A\subset S_{\ep^2}: \mu(A)&\ge \frac{\pi^{n/2}}{\Gamma(n/2+1)} \ep^{2(n-1)+1} \wedge \|f(\Ba)\|_2\le  C_0'(d^{9/4}+n^{\gamma/2-1/2})\sqrt{n}\ep^2 \quad  \forall \Ba\in A\Big)\\ 
&\le K_0^n (d^{9/4}+n^{\gamma/2-1/2})^n\ep + c_0^n.
\end{align*}

\end{theorem}





\subsection{Hypothetical assumption} For $\Bx\in S_{\ep^2}$, let $E_\Bx$ be the event that $\|f(\Bx)\|_2\le  C_0'(d^{9/4}+n^{\gamma/2-1/2})\sqrt{n}\ep^2$. Assume that  the following holds for all $\Bx\in S_{\ep^2}$ 

\begin{equation}\label{eqn:assumption}
\P(E_\Bx) = \P\Big(\|f(\Bx)\|_2\le  C_0'(d^{9/4}+n^{\gamma/2-1/2})\sqrt{n}\ep^2\Big) \le   {C_0''}^n (d^{9/4}+n^{\gamma/2-1/2})^n \ep^{2(n-1)},
\end{equation}

for some absolute constant $C_0''$. Then as 

$$\Vol(S_{\ep^2})=\frac{\pi^{n/2}}{\Gamma(n/2+1)}((1+2\ep^2)^{n}- (1-2\ep^2)^{n}) = O(\frac{n \pi^{n/2}}{\Gamma(n/2+1)} \ep^2),$$ 

one would have 

$$\int_{\Bx\in S_{\ep^2}} \P(E_\Bx) d\mu(\Bx) = O\Big(\frac{n \pi^{n/2}}{\Gamma(n/2+1)} {C_0''}^n (d^{9/4}+n^{\gamma/2-1/2})^n  \ep^{2n}\Big).$$

By using Markov's bound and Fubini, one thus infers that 

\begin{align*}
\P\Big(\mu\{\Bx \in S_{\ep^2}: E_\Bx \}\ge  \frac{ \pi^{n}}{\Gamma(n+1)} \ep^{2(n-1)+1}\Big)& \le n  {C_0''}^n (d^{1/4}+n^{\gamma/2-1/2})^n \ep^{2n} / \ep^{2(n-1)+1}\\ 
&= n {C_0''}^n (d^{9/4}+n^{\gamma/2-1/2})^n \ep.
\end{align*}

One would then be done with proving Theorem \ref{theorem:measure} by setting $K_0=2C_0''$.

However, the assumption \eqref{eqn:assumption} is not always true. Our next goal is to characterize those $\Bx$ with  $\P(E_\Bx) > {C_0''}^n (d^{9/4}+n^{\gamma/2-1/2})^n  \ep^{2(n-1)}$. For short, set 

\begin{equation}\label{eqn:M_d}
M_d:=C_0'(d^{9/4}+n^{\gamma/2-1/2}).
\end{equation}

Recall that  $E_\Bx$ is the event  $\|f(\Bx)\|_2\le  C_0'(d^{9/4}+n^{\gamma/2-1/2})\sqrt{n}\ep^2 = M_d \sqrt{n}\ep^2$. This is exactly a concentration event in a small ball. Fortunately, the latter has been studied extensively in the context of random matrix. In what follows we will introduce some key lemmas, our approach follows \cite{RV}.

\subsection{Diophantine Structure}
Let $y_1,\dots,y_m$ be real numbers. Rudelson and Vershynin \cite{RV} defined the 
essential {\em least common denominator} ($\LCD$) of $\By=(y_1,\dots,y_m)$  as follows. Fix parameters $\alpha$  and $\gamma_0$, where $\gamma_0 \in (0,1)$, and define
$$
\LCD_{\alpha,\gamma_0}(\By)
:= \inf \Big\{ D > 0: \dist(D \By, \Z^{m}) < \min (\gamma_0 \| D \By \|_2,\alpha) \Big\}.
$$

Here $\dist(\BA,\Z^m) := \inf_{\Ba\in \BA,\Bz\in \Z^m} \|\Ba-\Bz\|_2$. One typically assumes $\gamma_0$ to be a small constant. The inequality $\dist(D \By, \Z^m) < \alpha$ then yields that most coordinates of $\theta \Ba$ are within a small distance from 
non-zero integers.

\begin{theorem}\cite{RV}\cite[Theorem 3.3]{RV-rec}\label{theorem:RV}
  Consider a sequence $\By=(y_1,\ldots,y_m)$ of real numbers which satisfies
$\sum_{i=1}^m y_i^2 \ge 1$. Assume that $a_i$ are iid copies of $\xi$ satisfying \eqref{eqn:subgaussian}. Then, for every $\alpha > 0$ and $\gamma_0 \in (0,1)$, and for
  $$
  \ep \ge \frac{1}{\LCD_{\alpha,\gamma_0}(\By)},
  $$
  we have
  $$
   \sup_{y\in \C} \P_{a_1,\dots,a_m}\Big(|\sum_{1\le i\le m} a_i y_i-y| \le \ep\Big) \le C_1(\frac{\ep}{\gamma_0} + e^{-2\alpha^2}),
  $$

where $C_1$ is an absolute constant.
\end{theorem}

In application we will set $m=n^d$, while $\By_\Bx =(x_{i_1}\dots x_{i_d})_{1\le i_1,\dots,i_d \le n}$ and $a_{i_1\dots i_d}^{(l)}$  will play the role of $\By$ and of the $a_i$'s respectively.  As $\Bx\in S_{\ep^2}$, one has 

$$\|\By_\Bx\|^2_2 = \|\Bx\|_2^2 \ge (1-2\ep^2)^2 = 1- O(\ep^2).$$

We will choose $\gamma_0=1/2$ and
\begin{equation}\label{eqn:alpha}
\alpha:=
\begin{cases}
n^{7d/16-1/4} & \mbox{ if }   2\le d=o(\log n/\log \log n),\\
n^{d/4} & \mbox{ otherwise}.   
\end{cases} 
\end{equation}

Observe from Theorem \ref{theorem:RV} that if $(\LCD_{\alpha,\gamma_0}(\By_{\Bx}))^{-1} \le M_d\ep^2$ (with $M_d$ from \eqref{eqn:M_d}) then 

$$   \P_{a_1,\dots,a_m}\Big(|\sum_{1\le i\le m} a_i y_i| \le M_d\ep^2  \Big)  \le C_1 (2M_d \ep^2 + e^{-2\alpha^2}) \le 4C_1M_d \ep^2,$$

as one can check from \eqref{eqn:assumption:ep} that $\ep \ge n^{-(\gamma-1/3)n} \ge \exp(-n^{5/4}/2) \ge \exp(-\alpha^2/2)$.

Thus 

$$ \P(|f_i(\Bx)|\le M_d\ep^2) \le 4C_1 M_d \ep^2.$$

In fact, Theorem \ref{theorem:RV} also implies that $ \P(|f_i(\Bx)|\le M_d\delta^2) \le 4C_1 M_d \delta^2$ for any $\delta \ge \ep$.  Before proceeding further, we will need the following tenzorization trick.

\begin{lemma}\cite[Lemma 2.2]{RV}\label{lemma:tensor} Let $K,\delta_0$ be given.
Assume that $\P(|f_l(\Bx)|< \delta )\le K\delta^2$ for all $\delta\ge \delta_0$. Then 

$$\P(\|f(\Bx)\|_2<\delta \sqrt{n-1} )\le (C_0K\delta )^{n-1}$$

where $C_0$ is an absolute constant.
\end{lemma}

For the sake of completeness, we will present a short proof of Lemma \ref{lemma:tensor} in Appendix \ref{appendix:tensor}.

By independence and by  Lemma \ref{lemma:tensor}, we have 

\begin{align*}
\P(\|f(\Bx)\|_2 \le M_d n^{1/2} \ep^2)  &= \P(\sqrt{f_1^2(\Bx)+ \dots +f_{n-1}^2(\Bx)} \le M_d n^{1/2} \ep^2 )\\
& \le  (4C_0 C_1)^{n-1}(M_d\ep^2)^{n-1}\\
& \le K_0^{n-1}(M_d \ep^2)^{n-1},
\end{align*}
  
with $K_0:=4C_0 C_1$. Thus we have shown the following.
  
\begin{theorem}\label{theorem:largeLCD}
If $\Bx\in S_\ep^2$ and $(\LCD_{\alpha,1/2}(\By_\Bx))^{-1} \ge M_d \ep^ 2$, then 

$$\P(E_\Bx) = \P(\|f(\Bx)\|_2\le M_d\ep^2 n^{1/2}) \le K_0^{n-1} (M_d^2\ep^2)^{n-1}.$$
\end{theorem}

It remains to focus on $\Bx$ with relatively small $\LCD(\By_{\Bx})$,

\begin{equation}\label{eqn:x}
\LCD_{\alpha,1/2}(\By_{\Bx}) <  (M_d\ep^{2})^{-1}=:\ep'^{-1}.
\end{equation}

Thus the upper bound $M_d\sqrt{n}\ep^2$ in Claim \ref{claim:growth} becomes $\sqrt{n}\ep'$. The proof of Theorem \ref{theorem:measure} is complete if one can show the following.

\begin{theorem}\label{theorem:structure} There exists an absolute constant $c_0 \in (0,1)$ such that
$$\P\Big(\exists \Bx\in S_{\ep^2}: \LCD_{\alpha,1/2} (\By_{\Bx})\le \ep'^{-1} \wedge \|f(\Bx)\|_2\le  \sqrt{n} \ep'\Big) \le c_0^n.$$
\end{theorem}

Indeed, by Theorem \ref{theorem:structure}, with probability at least $1-c_0^n$, for all $\Ba\in S_{\ep^2}$ with $\LCD_{\alpha,1/2}(\By_\Ba) \le \ep'^{-1}$ one has $\|\Bf(\Ba)\|_2 > \sqrt{n} \ep'$. Conditioning on this event, all of the elements $\Ba$ of the set $A$ in Theorem \ref{theorem:measure} must have 
$\LCD_{\alpha,1/2}(\By_\Ba) \ge \ep'^{-1}$. But then the conclusion of Theorem \ref{theorem:measure} follows from Theorem \ref{theorem:largeLCD} via an application of Fubini and Makov's bound.

Before proving Theorem \ref{theorem:structure}, it is important to remark that if there exists $\Bx_0\in S_{\ep^2}$ satisfying $\LCD_{\alpha,1/2} (\By_{\Bx_0})\le \ep'^{-1}$ such that $\|f(\Bx_0)\|\le  \sqrt{n} \ep'$, then the normalized vector $\Bx_1=\Bx_0/\|\Bx_0\| \in S^{n-1}$ satisfies  

$$\LCD_{\alpha,1/2} (\By_{\Bx_1})\le (1+2\ep^2)\ep'^{-1} = (1+o(1)) \ep'^{-1} $$
and
\begin{align}\label{eqn:normalizing}
\|f(\Bx_1)\|_2\le  (1+2\ep^2)^d \sqrt{n} \ep' =(1+o(1))\sqrt{n} \ep',
\end{align}

where we used the assumption that $\ep$ is sufficiently small (recall from \eqref{eqn:assumption:ep} that $\ep \le d^{-n/4}$).

Hence it is enough to prove Theorem \ref{theorem:structure} for $\Bx\in S^{n-1}$ only. We next introduce two different types of vectors depending on their {\it sparsity}.

\begin{definition}\label{def:compressible}
Let $\delta,\rho\in(0,1)$ be sufficiently small (depending on $d$). A vector $\Bx\in \R^{n}$ is called {\it sparse} if $|supp(\Bx)|\le \delta n$. A vector $\Bx\in S^{n-1}$ is called {\it compressible} if $\Bx$ is within Euclidean distance $\rho$ from the set of all sparse vectors. A vector $\Bx\in S^{n-1}$ is called {\it incompressible} if it is not compressible. The sets of compressible and incompressible vectors will be denoted by $Comp(\delta,\rho)$ and $Incomp(\delta,\rho)$ respectively. 
\end{definition}
 
In what follows we will choose 

\begin{equation}\label{eqn:delta-rho}
\delta = \rho = \kappa_0/d^2,
\end{equation}

where $\kappa_0$ is a sufficiently small absolute constant .


\begin{theorem}\label{theorem:structure:compressible} There exists a positive constant $0<c_0<1$ such that the probability that there exists a compressible vector $\Bx \in Comp(\delta,\rho)$ with $\|f(\Bx)\|_2\le  (1+o(1))\ep' \sqrt{n}$ is bounded by $c_0^n$.
\end{theorem}

The proof of Theorem \ref{theorem:structure:compressible} will be presented in Section \ref{section:compressible}. Notice that this is where the error term $c_0^n$ arises in Theorem \ref{theorem:main}, which is unavoidable owing to Example \ref{example:1}. We remark further that Theorem \ref{theorem:structure:compressible} holds as long as $\ep'=o(1)$.


Our main analysis lies in the treatment for incompressible structural vectors.

\begin{theorem}\label{theorem:structure:incompressible}
Conditioning on the event considered in Theorem \ref{theorem:operatornorm}, the probability that there exists an incompressible $\Bx$ in $S_{\ep^2}$ with $\LCD_{\alpha,1/2}(\By_{\Bx})\le (1+o(1))\ep'^{-1}$ such that $\|f(\Bx)\|_2\le  (1+o(1))\ep'\sqrt{n}$ is bounded by $O(n^{-(1/16-o(1))n})$.
\end{theorem}

\section{Proof of Theorem \ref{theorem:structure:incompressible}}\label{section:incompressible}
 First of all, incompressible vectors spread out thanks to the following observation.
 
 \begin{fact}\cite[Lemma 3.4]{RV}\label{fact:spread} Let $\Bx\in Incomp(\delta,\rho)$. Then there exists a set $\sigma\subset \{1,\dots,n\}$ of cardinality $|\sigma|\ge \rho^2 \delta n/2$ such that 
 
 $$\frac{\rho}{\sqrt{2n}} \le |x_k|\le \frac{1}{\sqrt{\delta n}}, \forall k\in \sigma.$$
 \end{fact}

With the choice of $\delta$ and $\rho$ from \eqref{eqn:delta-rho}, 

 \begin{equation}\label{eqn:x_k}
 \frac{\kappa_0}{2d^2 \sqrt{n}} \le |x_k|\le \frac{d}{\kappa_0\sqrt{n}}.
 \end{equation}

As such, there are at least $\sigma^d$ product terms $x_{i_1}\dots x_{i_d}$ with $|x_{i_1}\dots x_{i_d}|\ge (\frac{\kappa_0}{2d^2})^d n^{-d/2} $. By definition, the $\LCD$ of $\By_\Bx$ is then at least  $n^{d/2}/(O(d))^{O(d)}$, where the implied constants depend on $\kappa_0$.  

We divide $[n^{d/2}/(O(d))^{O(d)},\ep'^{-1}]$ into dyadic intervals. For $n^{d/2}/(O(d))^{O(d)} \le D\le \ep'^{-1}$, define 

$$S_D:=\{\Bx\in S^{n-1},D\le \LCD_{\alpha,1/2}(\By_\Bx)\le 2D\}.$$

It follows from the definition of $\alpha$ from \eqref{eqn:alpha} that $\alpha \ll n^{d/2}/(O(d))^{O(d)} \le D$. Our next lemma is an upper bound for any fixed incompressible vector.

\begin{lemma}[Treatment for a single vector]\label{lemma:single}
Assume that $\Bx\in S_D$. Then for any $t>1/D$

$$\P(\|f(\Bx)\|_2 <t \sqrt{n} )\le (4C_0C_1)^{n-1}t^{2(n-1)}.$$
\end{lemma}

\begin{proof}(of Lemma \ref{lemma:single}) 
The claim follows from the definition of $\LCD_{\alpha,1/2}(\By_\Bx)$, Theorem \ref{theorem:RV} and Lemma \ref{lemma:tensor}.
\end{proof}




\subsection{Approximation by structure}
Recall that $\Bx\in S_D$ if  $D\le \LCD_{\alpha,1/2}(\By_\Bx)\le 2D$. Observe that $\By_\Bx$ is a vector in $\R^{n^d}$ with rich multiplicative structure. The main goal of this section is to translate this piece of diophantine information on $\By_\Bx$ to $\Bx$ itself.

\begin{lemma}[Nets of the level sets]\label{lemma:net}
There exists a $d^{O(d)}\alpha /D$-net $\CM_D$ of $S_D$ in the Euclidean metric of cardinality 
$$|\CM_D| \le  \binom{n}{d-1}  \left(1+ d^{O(d)} D/\alpha \right)^{d}  \times  \left(1 +  d^{O(d)} D/n^{d/2}  \right)^{n-d+1} .$$
\end{lemma}

\begin{proof}(of Lemma \ref{lemma:net}) By definition of $\LCD$,

$$\sum_{1\le i_1,\dots,i_d \le n} \|D(\Bx) x_{i_1}\dots x_{i_d}-p_{ij}\|_{\R/\Z}^2 \le \alpha^2$$ 

for some $D\le D(\Bx) \le 2D$. 

As there are $|\sigma|\ge \rho^2 \delta n/2$ indices $i$ satisfying \eqref{eqn:x_k}, by the pigeon-hole principle, there exist $d-1$ indices $i_1,\dots,i_{d-1}$ where $x_{i_j}$ satisfies \eqref{eqn:x_k} and such that 

\begin{align}\label{eqn:approx1}
\sum_{1\le j\le n}\|D(\Bx) x_{i_1}\dots x_{i_{d-1}} x_j\|_{\R/\Z}^2 \le \alpha^2/  (\rho^2 \delta/2)^d n^{d-1} &= \alpha^2 d^{6d} \kappa_0^{-3d} / 2^d n^{d-1}\nonumber \\
&:= f_0^2\alpha^2 /n^{d-1}  .
\end{align}

Without loss of generality, one assumes that $i_1=1,\dots,i_{d-1}=d-1$. Fix $x_1,\dots, x_{d-1}$ for the moment. Set 

$$D':=D(\Bx) x_1\dots x_{d-1}.$$ 

Then as $D\le D(\Bx)\le 2D$ and the $x_i$'s satisfy \eqref{eqn:x_k},

\begin{align}\label{eqn:D''}
&(\frac{\kappa_0}{2d^2})^{d-1}   \frac{D}{n^{(d-1)/2}}:= f_1  \frac{D}{n^{(d-1)/2}}    \le |D'| \le 2(\frac{d}{\kappa_0})^{d-1}  \frac{D}{n^{(d-1)/2}} :=f_2  \frac{D}{n^{(d-1)/2}}.
\end{align}

By definition and from \eqref{eqn:approx1}, with $\Bx':=(x_d,\dots,x_n)$, there exists $\Bp=(p_d,\dots,p_n)\in \Z^{n-d+1}$ such that 

$$\|D'\Bx' -\Bp\|_2 \le  f_0\alpha/n^{(d-1)/2}.$$

So 

\begin{align}\label{eqn:approx:1}
\| \Bx' - \frac{1}{D'}\Bp\|_2 \le f_0 \alpha /|D'| n^{(d-1)/2} & \le   (f_0/f_1)\alpha/D,
\end{align}

where we used the lower bound for $|D'|$ from \eqref{eqn:D''}.

Notice furthermore that 

\begin{align*}
\|\Bp\|_2 &\le \|D' \Bx'\|_2 +   f_0\alpha/n^{(d-1)/2}  \le |D'| +     f_0\alpha/n^{(d-1)/2}  \\
&\le  f_2 D/n^{(d-1)/2} +  f_0\alpha/n^{(d-1)/2} \\
&\le d^{O(d)}D/n^{(d-1)/2}. 
\end{align*} 



The collection $\CP$ of such integral vectors $\Bp$ has size at most 

$$|\CP| \le \left(1+ (d^{O(1)} D/n^{(d-1)/2})/\sqrt{n}\right)^{n-d+1} \le  \left(1 +  d^{O(d)} D/n^{d/2}  \right)^{n-d+1}.$$

Next, for the set $|z|\le  n^{(d-1)/2}/(f_1D)$ in $\R$ we choose an $\ep_d$-net $\CN_{local}$ with $\ep_d=(f_0/f_1) n^{(d-1)/2} \alpha/D(f_1 D + f_0\alpha)$. Clearly we can choose $\CN_{local}$ so that 

 $$|\CN_{local}| \le 1+2n^{(d-1)/2}/(\ep_df_1D)= 2(f_1/f_0)D/\alpha + 1 \le 1+ d^{O(d)} D/\alpha .$$

Define the following set in $\R^{n-d+1}$

$$\CN_{1\dots (d-1)}:= \{b\Bp, b\in \CN_{local}, \Bp\in \CP\}.$$

By definition,

\begin{equation}\label{eqn:N}
|\CN_{1\dots (d-1)}| \le  \left(1+ d^{O(d)} D/\alpha \right)  \times  \left(1 +  d^{O(d)} (D+\alpha)/n^{d/2}  \right)^{n-d+1} .
\end{equation}

Moreover, as $|1/D'|\le  n^{(d-1)/2}/(f_1D)$, there exists $b\in \CN_{local}$ such that $|1/D' -b|\le \ep_d$. As such, by \eqref{eqn:approx:1}

\begin{align*}
\|\Bx' - b \Bp\|_2 &\le \|\Bx' - \frac{1}{D'}\Bp\|_2 + \|(\frac{1}{D'}-b)\Bp\|_2\\
&\le (f_0/f_1) \alpha /D + \ep_d  ((f_1 D + f_0\alpha)/n^{(d-1)/2})\\ 
&\le 2 (f_0/f_1) \alpha/D.
\end{align*}
 
Thus $\CN_{1\dots (d-1)}$ is an $2 (f_0/f_1) \alpha/D$-net for $\Bx'=(x_d,\dots,x_n)$.

To continue, one  approximates $(x_1,\dots,x_{d-1})$ by an arbitrary $(f_0/f_1) \alpha/D$-net in $|z|\le 1$ of $\R^{d-1}$. We therefore obtain a net $\CN_{1\dots (d-1)}'$ that $3 (f_0/f_1) \alpha/D$-approximates the vector $(x_1,\dots,x_n)$, which has  size 

\begin{align*}
|\CN_{1\dots (d-1)}'| &\le (1+2(f_1/f_0)D/\alpha)^{d-1}  \times |\CN_{1\dots (d-1)}|\\
&\le  \left(1+ d^{O(d)} D/\alpha \right)^{d}  \times  \left(1 +  d^{O(d)} D/n^{d/2}  \right)^{n-d+1}.
\end{align*}

In summary, for each $d-1$ tuple $i_1,\dots,i_{d-1}$, one obtains a net $\CN_{i_1,\dots,i_d}'$ (by fixing $x_{i_1},\dots,x_{i_{d-1}}$ instead of $x_1,\dots,x_{d-1}$). The union set $\CM_D$ of all $\CN_{i_1,\dots,i_{d-1}}'$ will satisfy the conclusion of our theorem.
\end{proof}

\subsection{Passing from $\CM_D$ to $S_D$} Assume that there exists $\Bx$  with $D<\LCD(\By_\Bx) \le 2D$ such that $\|f(\Bx)\|_2\le \alpha \sqrt{n}/D$. Choose $\Bx_0\in \CM_D$ which is $3 (f_0/f_1) \alpha/D$-approximates $\Bx$. By conditioning on the event of Theorem \ref{theorem:operatornorm}, 

\begin{align*}
\|f(\Bx)\|_2 \le \|f(\Bx_0)\|_2 + (\sqrt{C_0} d^{1/4}\sqrt{n}) 3 (f_0/f_1) \alpha /D   &\le \alpha \sqrt{n}/D+  (\sqrt{C_0} d^{1/4}\sqrt{n}) 3 (f_0/f_1) \alpha /D\\  
&= (1+3\sqrt{C_0} d^{1/4}f_0/f_1)\alpha \sqrt{n}/D\\
&:= (f_3 \alpha/D) \sqrt{n}.
\end{align*}

On the other hand, it follows from Lemma \ref{lemma:single} and Lemma \ref{lemma:net} that 

\begin{align*}
& \P\Big(\exists \Bx_0 \in \CM_D, \|f(\Bx_0)\|_2   \le  (f_3 \alpha/D) \sqrt{n}\Big)  \le  (4C_0C_1)^{n-1}  (f_3 \alpha/D)^{n-1}|\CM_D| \nonumber \\
&\le   (O(1))^n  \binom{n}{d-1}  (f_3 \alpha/D)^{n-1} \left(1+ d^{O(d)} D/\alpha \right)^{d}  \left(1 +  d^{O(d)} D/n^{d/2}  \right)^{n-d+1} \nonumber \\
&=  (O(1))^n \binom{n}{d-1} \big(d^{O(d)}\big)^d  \left(\alpha d^{O(d)}/n^{d/2}+ \alpha d^{O(d)}/n^{d/2}\right)^{n-d-1} \left(1 +  d^{O(d)} D/n^{d/2}  \right)^{2}\nonumber \\ 
&\le  (O(1))^n \binom{n}{d-1}  d^{O(dn)} (n^{-d/16-1/4})^{n-d+1} (\ep')^{-2}.\nonumber\\
\end{align*}

Now we use the assumption that $\gamma/2 \le 19/36$. With this bound, $\ep \ge n^{-(\gamma/2-17/36)n} \ge n^{-1/18}$, and hence $\ep'=M_d\ep^2 > n^{-n/8}$. Thus, as long as $2\le d \le n^{\ep_0}$ for sufficiently small $\ep_0$, 

\begin{align}\label{eqn:remark:impro}
\P\Big(\exists \Bx_0 \in \CM_D, \|f(\Bx_0)\|_2 , \|f(\Bx_0)\|_2   \le  (f_3 \alpha/D) \sqrt{n}\Big)  &\le  d^{O(dn)} (n^{-d/16-1/4})^{n-d+1} (\ep')^{-2} \nonumber \\
&=O(n^{-n/16}).
\end{align}

In summary, we have shown that, conditioning on the the boundedness  of  the operator norm from Theorem \ref{theorem:operatornorm}, 

$$\P\Big(\exists \Bx: D<\LCD(\By_\Bx) \le 2D \wedge \|f(\Bx)\|_2 \le \alpha \sqrt{n}/D\Big) = n^{-n/16}.$$

Summing over the dyadic range $n^{d/2}/(O(d))^{O(d)} \le D \le \ep'^{-1}$ for $D$, one thus obtains

\begin{align*}
\P\left(\exists \Bx: \LCD(\By_\Bx) \le \ep'^{-1} \wedge \|f(\Bx)\|_2 \le n^{(d-1)/4} \sqrt{n}\ep'^{-1})\right) &\le O(n\log n) \times n^{-n/16},\\
&\le  n^{-(1/16-o(1)n)},
\end{align*}

completing the proof of Theorem \ref{theorem:structure:incompressible}.

\begin{remark}\label{remark:improvement}
Notice that in the last estimate of \eqref{eqn:remark:impro}, if $d$ is sufficiently large compared to $\gamma$, then  $d^{O(dn)} (n^{-d/16-1/4})^{n-d+1} (\ep')^{-2}$ is clearly at most $O(n^{-\Theta(n)})$ for any $\ep\ge n^{-(\gamma/2-17/36) n}$. Thus Theorem \ref{theorem:main} holds for any $\gamma=O(1)$ provided that $d$ is sufficiently large depending on $\gamma$. 
\end{remark}

\section{Control of the operator norm: proof of \eqref{eqn:operatornorm1} of Theorem \ref{theorem:operatornorm}}\label{section:operatornorm}

We will first prove a general statement which will be useful for the next section.
\begin{theorem}\label{theorem:operatornorm''}
Assume that $\xi$ is a sub-gaussian random variable with zero mean and unit variance satisfying \eqref{eqn:subgaussian}. Then there exists an absolute positive constant $C_0=C_0(K_0)$ independent of $d$ such that the following holds 

$$ \P\Big(\sup_{\Bx,\By,\dots, \Bz \in S^{n-1}}\sum_{1\le l\le n-1}(\sum_{i_1,\dots, i_d}a_{i_1i_2 \dots i_d}^{(l)}x_{i_1}y_{i_2}\dots z_{i_d})^2 \ge C_0\sqrt{d}n\Big)\le \exp(-dn).$$ 
\end{theorem}

Assuming this estimate for the moment, we now deduce Theorem \ref{theorem:operatornorm}.

\begin{proof}(of Theorem \ref{theorem:operatornorm}) The bound on $\sup_{\Bx\in S^{n-1}} \|\Bf_{rand}(\Bx)\|_2^2$ clearly follows from Theorem \ref{theorem:operatornorm''} by choosing $\By,\dots,\Bz$ to be $\Bx$.  For the gradient, we have 

\begin{equation}\label{eqn:norm:linear}
D_{l,\Bx,rand}^{(1)}(\By) = \sum_{1\le i_1,\dots,i_d \le n} a_{i_1\dots i_d}^{(l)} (y_{i_1}x_{i_2}\dots x_{i_{d-1}} x_{i_d} +\dots+ x_{i_1}x_{i_2}\dots x_{i_{d-1}} y_{i_d}).
\end{equation}

Thus by Cauchy-Schwarz inequality

\begin{align*}
\sum_l (D_{l,\Bx,rand}^{(1)}(\By))^2 \le d \Big[&\sum_l (\sum_{1\le i_1,\dots,i_d \le n} a_{i_1\dots i_d}^{(l)} y_{i_1}x_{i_2}\dots x_{i_{d-1}} x_{i_d})^2 + \dots + \\
+ &\sum_l (\sum_{1\le i_1,\dots,i_d \le n} a_{i_1\dots i_d}^{(l)} x_{i_1}x_{i_2}\dots x_{i_{d-1}} y_{i_d})^2\Big].
\end{align*}

By Theorem \ref{theorem:operatornorm''}, each summand is bounded by $C_0\sqrt{d}n$ with probability at least $1-\exp(-dn)$. Hence with probability at least $1-d\exp(-dn)\ge 1- \exp(-dn/2)$, 

$$\sum_l (D_{l,\Bx,rand}^{(1)}(\By))^2 \le C_0d^{5/2} n,$$

completing the bound for $\sup_{\Bx,\By}\|(D_{\Bx,rand}^{(1)}(\By))\|_2^2$.

The treatment for  $\sup_{\Bx,\By,\Bz \in S^{n-1}}\|D_{\Bx,rand}^{(2)}(\By,\Bz)\|_2^2$ is similar where in place of \eqref{eqn:norm:linear}, we write $D_{\Bx,rand}^{(2)}(\By,\Bz)$ as a sum of $O(d^2)$ summands. The upper bound $C_0d^{9/2}n$ can then be obtained again by applying Theorem \ref{theorem:operatornorm''} and Cauchy-Schwarz inequality. 
\end{proof}

What remains is to establish Theorem \ref{theorem:operatornorm''}. We first prove it for the case of fixed $\Bx,\By,\dots,\Bz$.

\begin{lemma}\label{lemma:operatornorm:xy}
Assume that $\Bx,\By,\dots,\Bz \in S^{n-1}$. Then there exists an absolute positive constant $C_0=C_0(K_0)$ such that 

$$ \P\Big(\sup_{\Bx,\By,\dots, \Bz \in S^{n-1}}\sum_{1\le l\le n-1}(\sum_{i_1,\dots, i_d}a_{i_1i_2 \dots i_d}^{(l)}x_{i_1}y_{i_2}\dots z_{i_d})^2 \ge C_0\sqrt{d} n \Big)\le  \exp(-16d n) .$$
\end{lemma}

\begin{proof}(of Lemma \ref{lemma:operatornorm:xy})
We observe that for any $l$, $(\sum_{i_1,\dots, i_d}a_{i_1i_2 \dots i_d}^{(l)}x_{i_1}y_{i_2}\dots z_{i_d})^2$ is a sub-exponential random variable with mean one and bounded variance. Lemma \ref{lemma:operatornorm:xy} then follows by a standard deviation result.
\end{proof}

We now extend the result above to the case $\By,\dots,\Bz$ are fixed.

\begin{lemma}\label{lemma:operatornorm:x}
Assume that $\By,\dots,\Bz$ are fixed unit vectors of $S^{n-1}$, then

$$\P\Big(\sup_{\Bx\in S^{n-1}}\sum_{1\le l\le n-1}(\sum_{i_1,\dots, i_d}a_{i_1i_2 \dots i_d}^{(l)}x_{i_1}y_{i_2}\dots z_{i_d})^2 \ge C_0\sqrt{d}n\Big)\le \exp(-(16d-6)n).$$ 

\end{lemma}

\begin{proof}(of Lemma \ref{lemma:operatornorm:x}) Consider an $1/2$-net $\CN$ of $S^{n-1}$. We first claim that 

\begin{align}\label{eqn:netapprox}
& \P\Big(\sup_{\Bx\in S^{n-1}} \sum_{1\le l\le n-1}(\sum_{i_1,\dots, i_d}a_{i_1i_2 \dots i_d}^{(l)}x_{i_1}y_{i_2}\dots z_{i_d})^2 \ge M^2\Big) \nonumber \\
& \le \P\Big(\sup_{\Bx\in \CN} \sum_{1\le l\le n-1}(\sum_{i_1,\dots, i_d}a_{i_1i_2 \dots i_d}^{(l)}x_{i_1}y_{i_2}\dots z_{i_d})^2 \ge (M/2)^2\Big).
\end{align}

For simplicity, consider the matrix $A_{\By,\dots,\Bz}:=(a_{li_1}(\By,\dots,\Bz))_{1\le l\le n-1,1\le i_1\le n}$, where $a_{li_1}(\By,\dots,\Bz):=\sum_{i_1,\dots, i_d}a_{i_1i_2 \dots i_d}^{(l)}y_{i_2}\dots z_{i_d}$ . It then follows that 

$$\sum_{1\le l\le n-1}(\sum_{i_1,\dots, i_d}a_{i_1i_2 \dots i_d}^{(l)}x_{i_1}y_{i_2}\dots z_{i_d})^2 = \|A_{\By,\dots,\Bz}\Bx\|_2^2.$$

Now assume that $\sup_{\Bx\in S^{n-1}} \|A_{\By,\dots,\Bz}\Bx\|_2=\|A_{\By,\dots,\Bz}\|_{op}$ is attained at $\Bx=(x_1,\dots,x_n)$. Choose $\Bx'\in \CN$ such that $\|\Bx-\Bx'\|_2\le 1/2$. By definition, as $A_{\By,\dots,\Bz}$ is a linear operator, 

$$\|A_{\By,\dots,\Bz}\Bx-A_{\By,\dots,\Bz}\Bx'\|_2=\|A_{\By,\dots,\Bz}(\Bx-\Bx')\|_2 \le \|\Bx-\Bx'\|_2 \|A_{\By,\dots,\Bz}\|_{op}\le \frac{1}{2}\|A_{\By,\dots,\Bz}\|_{op}.$$

By the triangle inequality, it is implied that 

$$\|A_{\By,\dots,\Bz}\Bx'\|_2 \ge \frac{1}{2}\|\MA\|_{op},$$

proving our claim.

To conclude the proof, notice that $S^{n-1}$ has an $1/2$-net $\CN$ of size at most  $2n 5^n$. We then apply Lemma \ref{lemma:operatornorm:xy} and the union bound

$$\P\Big(\sup_{\Bx\in S^{n-1}} \sum_{1\le l\le n-1}(\sum_{i_1,\dots, i_d}a_{i_1i_2 \dots i_d}^{(l)}x_{i_1}y_{i_2}\dots z_{i_d})^2 \ge Cn\Big) \le 2n 5^n \times \exp(-16dn) \le   \exp(-(16d -6)n).$$

\end{proof}

Observe that one can also extend \eqref{eqn:netapprox} to the case that $\Bx,\By$ vary, 

\begin{align}\label{eqn:quad:net}
&\P\Big(\sup_{\Bx,\By\in S^{n-1}}\sum_{1\le l\le n-1}(\sum_{i_1,\dots, i_d}a_{i_1i_2 \dots i_d}^{(l)}x_{i_1}y_{i_2}\dots z_{i_d} )^2 \ge M^2\Big) \nonumber \\
&\le \P\Big(\bigvee_{\By\in \CN} \sup_{\Bx \in S^{n-1}}\sum_{1\le l\le n-1}|\sum_{i_1,\dots, i_d}a_{i_1i_2 \dots i_d}^{(l)}x_{i_1}y_{i_2}\dots z_{i_d}|^2 \ge (M/2)^2\Big).
\end{align}

Thus one obtains the following analog of Lemma \ref{lemma:operatornorm:x} when $\Bx$ and $\By$ are not fixed.

\begin{align*}
& \P\Big(\sup_{\Bx,\By \in S^{n-1}} \sum_{1\le l\le n-1}(\sum_{i_1,\dots, i_d}a_{i_1i_2 \dots i_d}^{(l)}x_{i_1}y_{i_2}\dots z_{i_d})^2 \ge Cn\Big) \\
& \le 2n 5^n \times  \P\Big(\sup_{\Bx \in S^{n-1}}\sum_{1\le l\le n-1}(\sum_{i_1,\dots, i_d}a_{i_1i_2 \dots i_d}^{(l)}x_{i_1}y_{i_2}\dots z_{i_d})^2 \ge (M/2)^2\Big)\\
& \le 2n5^n \exp(-(16d -6)n) \le \exp(-(16d -12)n). 
\end{align*}

To conclude the proof of Theorem \ref{theorem:operatornorm''}, one just  iterates the argument above $d$ times.  Finally, we remark that Theorem \ref{theorem:operatornorm''} yields the following more general looking version.

\begin{theorem}\label{theorem:operatornorm'} Assume that $1\le k\le n$, and that $\MA=\{a_{i_1,\dots, i_d}^{(l)}, 1\le i_1,\dots,i_d \le k, 1\le l\le n-1 \}$ is an array of iid random copies of a subgaussian random variable $\xi$ of zero mean and unit variance satisfying \eqref{eqn:subgaussian}. Then there exists a positive constant $C_0$ such that the following holds 

$$\P\Big(\sup_{\Bx,\By,\dots,\Bz\in S^{k-1}}\sum_{1\le l\le n-1}|\sum_{i_1,\dots, i_d}a_{i_1i_2 \dots i_d}^{(l)}x_{i_1}y_{i_2}\dots z_{i_d}|^2 \ge C_0\sqrt{d}n\Big)\le  \exp(-dn).$$ 
\end{theorem}

\section{Control of compressible vectors} \label{section:compressible}
We will prove a more general estimate as follows.

\begin{theorem}\label{theorem:structure:compressible'} With sufficiently small constant $c_{sparse}$,
$$\P\Big(\inf_{\Bx\in Comp(\delta,\rho)}  \sum_{1\le l\le n-1}|f_l(\Bx)|^2 \le c_{sparse} n\Big) \le c_0^n.$$
\end{theorem}

Recall from \eqref{eqn:delta-rho} that  $\delta = \rho = \kappa_0/d^2$ for a sufficiently small absolute constant  $\kappa_0$. In order to prove Theorem \ref{theorem:structure:compressible'}, we will need to work with rectangular arrays.

\begin{theorem}\label{theorem:rectangular}
Assume that $\MA=\{a_{i_1\dots i_d}^{(l)}, 1\le i_1,\dots,i_d \le k, 1\le l\le n-1 \}$ is an array of iid random copies of $\xi$ satisfying \eqref{eqn:subgaussian}, with $k=\delta n$. Then there exist absolute constants $c_1,c_2$ such that the following holds

$$\P\Big(\inf_{\Bx\in S^k} \sum_{1\le l\le n-1}|f_l(\Bx)|^2\le c_1n\Big) \le \exp(-c_2n).$$
\end{theorem}

Indeed we shall prove a slightly stronger result as below.

\begin{theorem}[Rectangular case for multilinear  forms]\label{theorem:rectangular:bilinear}
 With the same assumption as in Theorem \ref{theorem:rectangular}, there exist positive constants $c_1,c_2$ such that the following holds 

$$\P\Big(\inf_{\Bx,\By,\dots,\Bz\in S^k} \sum_{1\le l\le n}(a_{i_1\dots i_d}^{(l)}x_{i_1}y_{i_2}\dots z_{i_d})^2\le c_1n\Big) \le \exp(-c_2n).$$
\end{theorem}

In order to prove Theorem \ref{theorem:rectangular}, we first need the following easy result of non-concentration (see for instance \cite[Lemma 2.6]{RV}).

\begin{claim}\label{claim:nonconcentration}  There exists $\mu\in(0,1)$ such that for for any $(a_1,\dots,a_N)\in S^{N-1}$, the random sum $S=\sum \xi_i a_i$, where $\xi_1,\dots,\xi_N$ are independent copies of $\xi$ from \eqref{eqn:subgaussian}, satisfies

$$\P(|S|\le 1/2)\le \mu.$$
\end{claim}

We recall an analog of the tesorization lemma from Section \ref{section:incompressible}.

\begin{lemma}\label{lemma:tensorization} Let $\eta_1,\dots,\eta_n$ be independent non-negative random variable, and let $K,\delta \ge 0$.

\begin{itemize}
\item Assume that for each $l$, $\P(\eta_l<\ep)\le K\ep$ for all $\ep\ge \delta$. Then

$$\P(\sum \eta_l^2 <\ep ^2 n)\le (C_0K\ep)^n$$

for all $\ep \ge \delta$.
\vskip .1in
\item Consequently, assume that there exist $\lambda$ and $\mu\in (0,1)$ such that for each $l$, $\P(\eta_l<\lambda)\le \mu$. Then there exist $\lambda_1>0$ and $\mu_1\in (0,1)$ depending on $\lambda,\mu$ such that 

$$\P(\sum \eta_l^2 <\lambda_1 n)\le \mu_1^n.$$
\end{itemize}

\end{lemma}

As $\sum_{1\le i_1,i_2 \dots,i_d \le n} (x_{i_1}y_{i_2} \dots z_{i_d})^2 =1$, it follows from Claim \ref{claim:nonconcentration} and Lemma \ref{lemma:tensorization} the following analog of Theorem \ref{theorem:rectangular}.

\begin{lemma}[Estimate for fixed compressible vectors]\label{cor:xy} With the same assumption as in Theorem \ref{theorem:rectangular}, and let $\Bx,\By,\dots,\Bz$ be fixed. Then there exist constants $\eta,\nu \in (0,1)$  such that 

$$\P\Big(\sum_{1\le l\le n-1} (\sum_{1\le i_1,\dots,i_d \le n}a_{i_1\dots i_d}^{(l)}x_{i_1}y_{i_2} \dots z_{i_d})^2 <\eta n\Big)\le \nu^n.$$

\end{lemma}

Similarly to our treatment of the operator norm in the previous section, we can improve the above as follows.

\begin{theorem}\label{cor:x} With the same assumption as in Theorem \ref{theorem:rectangular}, and let $\By,\dots,\Bz \in S^k$ be fixed. Then there exist constants $\eta,\nu \in (0,1)$ such that 

$$\P\Big(\inf_{\Bx\in S^k}\sum_{1\le l\le n-1} (\sum_{1\le i_1,\dots,i_d \le n} a_{i_1\dots i_d}^{(l)}x_{i_1}y_{i_2}\dots z_{i_d})^2 <4\eta n\Big)\le \nu^{(1-o(1))n}.$$

\end{theorem}

For short, we denote $\sum_{1\le l\le n-1} (\sum_{1\le i_1,\dots,i_d \le n} a_{i_1\dots i_d}^{(l)}x_{i_1}y_{i_2}\dots z_{i_d})^2$ by $\|A_{\By,\dots,\Bz}(\Bx)\|_2^2$, emphasizing that this operator depends on $\By,\dots,\Bz$.

\begin{proof}(of Theorem  \ref{cor:x}) Let $\alpha_d=\alpha_0d^{-3/2}$ with sufficiently small $\alpha_0$ to be chosen. It is simple to show that (see for instance \cite[Lemma 2.6]{MS} or \cite[Proposition 2.1]{RV-rec}) there exists an $\alpha_d$-net $\CN$ in $S^k$ of cardinality at most $|\CN|\le 2k(1+\frac{2}{\alpha_d})^k$. Let $\eta,\nu$ be the numbers in Corollary \ref{cor:xy}, by the union bound,

\begin{align*}
\P\Big(\exists \Bx\in \CN: \|A_{\By,\dots,\Bz}(\Bx)\|_2^2<\eta n\Big)&= \P\Big(\exists \Bx\in \CN: \sum_{1\le l\le n-1} |\sum_{1\le i_1,\dots,i_d \le n} a_{i_1\dots i_d}^{(l)}x_{i_1}y_{i_2}\dots z_{i_d}|^2 <\eta n\Big)\\
& \le  2k(1+\frac{2}{\alpha_d})^k \nu^n\\
& \le (\kappa_0 n/d^2) (1+ 2 d^{3/2}/\alpha_0)^{\kappa_0n/d^2} \nu^n \\
&\le \nu^{(1-o(1))n},
\end{align*}

where we used the fact that $\kappa_0$ is sufficiently small (compared to $\alpha_0$) and $2\le d=o(n)$.

Within this event, let $\Bx$ be any unit vector in $S^k$. Choose a point $\Bx'\in \CN$ such that $\|\Bx -\Bx'\|_2 \le \alpha_d$. By Theorem \ref{theorem:operatornorm'}, with probability at least $1-\exp(-dn)$ we have 

$$\|A_{\By,\dots,\Bz}(\Bx-\Bx')\|_2< \alpha_d \sqrt{C_0} d^{1/4}\sqrt{n} = \alpha_0 \sqrt{C_0} d^{-5/4} \sqrt{n} \le \sqrt{\eta n},$$

where we chose $\alpha_0$ so that $\alpha_0 \sqrt{C_0}\le \sqrt{\eta}$. It thus follows that

$$\|A_{\By,\dots,\Bz} \Bx'\|_2\le  \sqrt{\eta n} + \sqrt{\eta n}= 2\sqrt{\eta n},$$ 

completing the proof.
\end{proof}

\begin{proof}(of Theorem \ref{theorem:rectangular:bilinear}) Iterate the argument above $d$ times by fixing lesser terms at each step, one arrives at the conclusion of Theorem \ref{theorem:rectangular:bilinear}, noting that the entropy loss is at most (taking into account the number of $\alpha_d$-nets for all $\Bx,\By,\dots,\Bz$)

$$\Big((\kappa_0n/d^2) (1+ 2 d^{3/2}/\alpha_0)^{\kappa_0 n/d^2}\Big)^d  = (\kappa_0n/d^2)^d (1+ 2 d^{3/2}/\alpha_0)^{\kappa_0 n/d} \nu^n \le \nu^{(1-o(1))n},$$

again provided that $\kappa_0$ is sufficiently small compared to $\alpha_0$ and $2\le d =o(n)$.
\end{proof}

We now deduce Theorem \ref{theorem:structure:compressible'} in the same manner.




 

\begin{proof}(of Theorem \ref{theorem:structure:compressible'}) By Theorem \ref{theorem:operatornorm'}, with probability at least $1-\exp(-n)$ we have the following for any pair $\Bx,\Bx'$ with $\|\Bx-\Bx'\|_2\le \rho$, 

\begin{align*}
\|A_{\Bx,\Bx,\dots,\Bx}\Bx - A_{\Bx',\Bx' \dots,\Bx'}\Bx'\|_2 &\le \|A_{\Bx,\Bx,\dots,\Bx}\Bx - A_{\Bx,\Bx,\dots,\Bx}\Bx'\|_2+ \|A_{\Bx,\Bx,\dots,\Bx}\Bx' - A_{\Bx',\Bx,\dots,\Bx}\Bx'\|_2+\dots + \\
&+   \|A_{\Bx',\Bx',\dots,\Bx}\Bx' - A_{\Bx',\Bx',\dots,\Bx'}\Bx'\|_2 \\ 
&\le d \rho \sqrt{C_0} d^{1/4} \sqrt{n}\\
&\le \sqrt{c_1 n},
\end{align*}

where we used the fact that $\rho=\kappa_0/d^2$ with sufficiently small $\kappa_0$ compared to $c_1$. 

As $\Bx'$ ranges over vectors of $S^{n-1}$ of support at most $k=\delta n$, an application of Theorem \ref{theorem:rectangular} implies that 

$$\P(\inf_{\Bx\in Comp(\delta,\rho)}  \sum_{1\le l\le n-1}|f_l(\Bx)|^2 \le 2c_3 n)\le \binom{n}{k} \exp(-c_2n)  \le c_0^n,$$

for some $0<c_0<1$, completing the proof of Theorem \ref{theorem:structure:compressible'}.

\end{proof}

\appendix

\section{proof of Lemma \ref{lemma:tensor}}\label{appendix:tensor}
We restate the lemma.

\begin{lemma}[Lemma \ref{lemma:tensor}]Let $K,\delta_0\ge 0$ be given.
Assume that $\P(|X_1|< \delta )\le K\delta$ for all $\delta\ge \delta_0$. Then 

$$\P(X_1^2+\dots+X_n^2  <\delta n)\le (C_0K\delta)^n.$$
\end{lemma}

\begin{proof} Assume that $\delta\ge \delta_0$. By Chebyshev's inequality

$$\P({X_1}^2+\dots+{X_n}^2 \le \delta n) \le \E \exp( n- \sum_{i=1}^n {X_i}^2/\delta)=\exp(n) \prod_{i=1}^n \E \exp(-{X_i}^2/\delta).$$

On the other hand,

$$ \E \exp(-{X_i}^2/\delta) = \int_0^1 \P(\exp(-X_i^2/\delta)>s) ds = \int_0^\infty 2u \exp(-u^2) \P(X_i<\delta u)du.$$

For $0\le u\le 1$ we use $\P(X_i \le \delta u)\le \P(X_i\le \delta) \le K \delta$, while for $u\ge 1$ we have $\P(X_i \le \delta u)\le K \delta u$. Thus

$$ \E \exp(-{X_i}^2/\delta)  =  \int_0^1 2u \exp(-u^2) K \delta  du +  \int_1^\infty 2u \exp(-u^2) K \delta u du \le C_0 K \delta.$$
\end{proof}

{\bf Acknowledgement}. The author would like to thank R.~Vershynin and V.~Vu for discussions at the early stage of this work. He is also grateful to the anonymous reviewer for helpful comments and suggestions that help correct and improve the presentation of the manuscript.

\end{document}